\documentclass[12pt]{amsart}

\usepackage{amssymb}
\usepackage{amsmath}
\numberwithin{equation}{section}
\usepackage{color}
\usepackage{graphicx}

\def\cx{{\mathcal X}}

\def\cl{{\mathcal L}}
\def\cc{{\mathcal C}}
\def\cm{{\mathcal M}}
\def\zz{{\mathbb Z}}
\def\nn{{\mathbb N}}

\def\chl{{\sf chiefl}}
\def\mml{{\sf minmaxl}}
\def\ml{{\sf modl}}
\def\nl{{\sf nl}}

\newtheorem{thm}{Theorem}[section]
\newtheorem{lemma}[thm]{Lemma}
\newtheorem{cor}[thm]{Corollary}
\newtheorem{prop}[thm]{Proposition}

\theoremstyle{definition}

\begin{document}

\title[Subgroup lattice characterization]{A new subgroup lattice characterization of finite solvable groups}         
\author[Shareshian]{John Shareshian$^1$}
\address{Department of Mathematics, Washington University, St. Louis, MO 63130}
\thanks{$^{1}$Supported in part by NSF Grant
DMS-0902142}
\email{shareshi@math.wustl.edu}

\author[Woodroofe]{Russ Woodroofe}
\address{Department of Mathematics, Washington University, St. Louis, MO 63130}
\email{russw@math.wustl.edu}

\date{}          

\begin{abstract}
We show that if $G$ is a finite group then no chain of modular elements in its subgroup lattice $\cl(G)$ is longer than a chief series.  Also, we show that if $G$ is a nonsolvable finite group then every maximal chain in $\cl(G)$ has length at least two more than the chief length of $G$, thereby providing a converse of a result of J. Kohler.  Our results enable us to give a new characterization of finite solvable groups involving only the combinatorics of subgroup lattices.  Namely, a finite group $G$ is solvable if and only if $\cl(G)$ contains a maximal chain $\cx$ and a chain $\cm$ consisting entirely of modular elements, such that $\cx$ and $\cm$ have the same length.
\end{abstract}

\maketitle

\section{Introduction}

\subsection{Main results} \label{mrsec}

Given a finite group $G$, let $\cl(G)$ be the subgroup lattice of $G$. We write $\mml(G)$ for the minimum length of a maximal chain in $\cl(G)$, $\chl(G)$ for the length of a chief series of $G$, and $\ml(G)$ for the maximum length of a chain of modular elements in $\cl(G)$. (The definition of a modular element in a lattice is given in Section \ref{defsec}.)  Our concluding result is the following characterization of finite solvable groups, using only the combinatorial structure of subgroup lattices.  As discussed below, our theorem is not the first such characterization.

\begin{thm} \label{concl}
Let $G$ be a finite group.  Then $G$ is solvable if and only if $\mml(G)=\ml(G)$.
\end{thm}

It is not hard to see that a group $G$ is finite if and only if $\cl(G)$ is finite.  Indeed, any infinite group has either an infinite cyclic subgroup or infinitely many finite cyclic subgroups.  Therefore, Theorem \ref{concl} distinguishes finite solvable groups from all other groups.  

Every normal subgroup of $G$ is a modular element of $\cl(G)$.  It follows immediately that $\ml(G) \geq \chl(G)$.  Our next result, which is proved in Section \ref{modproof}, allows us to prove Theorem \ref{concl} by comparing $\chl(G)$ and $\mml(G)$.

\begin{thm} \label{mod}
For every finite group $G$, we have $\ml(G)=\chl(G)$.  In other words, no chain of modular elements in $\cl(G)$ is longer than a chief series.
\end{thm}

If ${\cm}$ is a chain of modular elements in a lattice $L$ and $\cc$ is any other chain in $L$, then the sublattice of $L$ generated by $\cm$ and $\cc$ is a distributive lattice, and therefore graded.  (See for example \cite[2.1]{St1}, and its proof.)  It follows that $\mml(G) \geq \chl(G)$ for all finite groups $G$.  Thus the fact that $\mml(G)=\chl(G)$ when $G$ is solvable follows from the following result of J. Kohler.  (In fact, Kohler proves a stronger theorem involving indices.)

\begin{thm}[See Theorem 1 of \cite{Ko}] \label{solv}
If $G$ is a finite solvable group, then $\cl(G)$ contains a maximal chain whose length is the same as that of a chief series for $G$.
\end{thm}

Note that Theorem \ref{mod} for solvable groups follows from Theorem \ref{solv}.  Indeed, let $\cm$ be a maximal chain in $\cl(G)$ having length $\chl(G)$.  If $\cc$ is a chain of modular elements in $\cl(G)$ then, as $\cc$ and $\cm$ together generate a graded lattice and $\cm$ is maximal, we see that $\cc$ is not longer than $\cm$.

In Section \ref{nsproof}, we prove the following result, which when combined with Theorems \ref{mod} and \ref{solv}, proves Theorem \ref{concl}.

\begin{thm} \label{nonsolv}
If $G$ is a nonsolvable finite group then $\mml(G) \geq \chl(G)+2$.
\end{thm}

%


It follows from Theorems \ref{solv} and \ref{nonsolv} that there is no finite group $H$ satisfying $\mml(H)=\chl(H)+1$.  We show in Section \ref{chdiffex} that for each $k \geq 2$, there exists some group $H$ such that $\mml(H)=\chl(H)+k$.  Indeed, after noting that $\mml(A_5)=3$ we produce, for each $n \geq 1$, a direct product $G_n$ of $n$ pairwise nonisomorphic simple groups satisfying $\mml(G_n)=2n+2$.


\subsection{Background and motivation} \label{histsec}

As mentioned above, Theorem \ref{concl} is not the first characterization of finite solvable groups that uses only the combinatorial structure of subgroup lattices.  Before describing earlier characterizations, we remark that long before they were discovered, it was shown independently by M. Suzuki in \cite{Su} and by G. Zacher in \cite{Za} that if exactly one of the finite groups $G,H$ is solvable then $G,H$ do not have isomorphic subgroup lattices.   In \cite{Ja}, B. V. Jakovlev proves that the same result holds without the condition that $G$ is finite.  

In \cite{Sch1}, R. Schmidt proves the following result. (This result is described in English in \cite[Theorem~5.3.5]{Sch2}.)

\begin{thm}[R. Schmidt] \label{schm}
Let $G$ be a finite group.  The following conditions on $G$ are equivalent.
\begin{enumerate}
  \item $G$ is solvable.
  \item There exists a chain $1=N_0<\ldots<N_r=G$ in $\cl(G)$ such that each $N_i$ is modular in $\cl(G)$ and, for each $i \in [r]$, the interval $[N_{i-1},N_i]$ in $\cl(G)$ is a modular lattice.
  \item There exists a chain $1=H_0<\ldots<H_s=G$ in $\cl(G)$ such that, for each $i \in [s]$, $H_{i-1}$ is modular in $\cl(H_i)$, and the interval $[H_{i-1},H_i]$ is a modular lattice.
  \item There exists a chain $1=S_0<\ldots<S_t=G$ in $\cl(G)$ such that, for each $i \in [t]$, $S_{i-1}$ is a maximal subgroup of $S_i$ and is modular in $\cl(S_i)$.
\end{enumerate}
\end{thm}

Note that if $G$ is solvable, then a chief series for $G$ satisfies condition (2) of Theorem \ref{schm} and a composition series for $G$ satisfies condition (4).  On the other hand, if $G$ is not solvable, then no normal series for $G$ satisfies (2), and no subnormal series satisfies (3) or (4).  It is not the case that every modular element of $\cl(G)$ is normal in $G$.  Indeed, upon considering $\cl(S_3)$, we see that it is in general impossible to discern solely from the combinatorial structure of $\cl(G)$ whether a modular element is normal.  One might view the key point of Theorem \ref{schm} to be that, when attempting to derive facts about the algebraic structure of $G$ from the combinatorial structure of $\cl(G)$, one can sometimes get away with ignoring the difference between modular elements and normal subgroups.  The reduction of Theorem \ref{concl} to Theorems \ref{solv} and \ref{nonsolv} through Theorem \ref{mod} is another example of this point at work.

Another motivating factor in our study of chains in subgroup lattices is the hope of developing a nongraded analogue of the theory of supersolvable lattices.  The following characterization of solvable finite groups, due to the present authors, will be pertinent to our brief discussion of these lattices.

\begin{thm}[See \cite{Sh,Wo}] \label{shwo}
Let $G$ be a finite group.  The following conditions on $G$ are equivalent.
\begin{itemize}
  \item[(a)] $G$ is solvable.
  \item[(b)] $\cl(G)$ admits an EL-labeling.
  \item[(c)] The order complex of $\cl(G)$ is (nonpure) shellable.
\end{itemize}
\end{thm}

We will not define the terms ``EL-labeling", ``order complex" and ``(nonpure) shellable" here, as we will not need them outside of the present discussion.  These terms are defined in \cite{BjWa}.   The equivalence of (a) and (c) is proved in \cite{Sh}.  Every poset admitting an EL-labeling has a shellable order complex, and the fact that (a) implies (b) is proved in \cite{Wo}.

Supersolvable lattices were first defined and studied by R. Stanley in \cite{St1}.  A finite lattice $L$ is called supersolvable if it contains a maximal chain $\cm$ such that the sublattice of $L$ generated by $\cm$ and any other chain $\cc$ is distributive.  Such a chain $\cm$ is called an $M$-chain.  As noted above, a maximal chain consisting entirely of modular elements is an $M$-chain.  Thus, a finite group $G$ is supersolvable if and only if a chief series for $G$ is an $M$-chain in $\cl(G)$ (hence the name ``supersolvable lattice").  In addition to subgroup lattices of supersolvable finite groups, there are many other interesting classes of supersolvable lattices.  These lattices have received considerable attention (see for example \cite{St2,Bj,JaTe,Te,McN,Th}).  We mention in particular the work of A. Bj\"orner in \cite{Bj}, where it is shown that a supersolvable lattice admits an EL-labeling.

Since distributive lattices are graded, supersolvable lattices are also graded.  It follows from Bj\"orner's work that if $G$ is a supersolvable finite group then the order complex of $\cl(G)$ is shellable.  On the other hand, K. Iwasawa showed in \cite{Iw} that $\cl(G)$ is graded (and therefore has pure order complex) if and only if $G$ is supersolvable.  Thus, in \cite{Bj}, the question of shellability when $\cl(G)$ is graded is settled by showing that the subgroup lattices in question belong to a larger class of lattices whose combinatorial structure guarantees their shellability.  The contribution of \cite{Wo} to Theorem \ref{shwo} is an extension Bj\"orner's work to nongraded subgroup lattices, but this extension requires a close examination of the algebraic structure of solvable groups. It would be interesting to extend the ideas of Stanley and Bj\"orner, by finding a large class ${\mathcal S}$ of lattices such that
\begin{itemize}
  \item if $G$ is a finite solvable group then $\cl(G) \in {\mathcal S}$,
  \item ${\mathcal S}$ contains interesting members that are not subgroup lattices, and
  \item combinatorics alone guarantees that each $L \in {\mathcal S}$ admits an EL-labeling.
\end{itemize}

We can rephrase Iwasawa's result to say that $G$ is supersolvable if and only if {\it every} maximal chain in $\cl(G)$ has length $\ml(G)$.  Theorem \ref{concl} says that $G$ is solvable if and only if {\it some} maximal chain in $\cl(G)$ has length $\ml(G)$, and can be seen as a nongraded analogue of Iwasawa's result.  Perhaps this is a first step towards finding a good definition of a ``solvable lattice".  
Interesting previous work involving nongraded analogues of supersolvability appears in \cite{BlSa,LiSa,McNTh}.  However, there exist finite solvable groups (such as $S_4$) whose subgroup lattices are not of the type studied therein.

\section{Definitions and notation} \label{defsec}

All groups and lattices discussed here are assumed to be finite.  For a group $G$, $\cl(G)$ will denote the set of all subgroups of $G$, partially ordered by inclusion.  Then $\cl(G)$ is a lattice, with respective meet and join operations $H \wedge K=H \cap K$, $H \vee K=\langle H,K \rangle$.

A {\it chain} in $\cl(G)$ is a subset that is totally ordered by the inclusion relation.  Such a chain is {\it maximal} if it is not properly contained in any other chain.  Every maximal chain contains $1$ and $G$.

An element $m$ of a lattice $L$ is {\it modular} in $L$ if 
\begin{itemize}
\item $x \vee (m \wedge y)=(x \vee m) \wedge y$ for all $x,y \in L$ satisfying $x \leq y$, and
\item $m \vee (x \wedge n)=(m \vee x) \wedge n$ for all $x,n \in L$ satisfying $m \leq n$.
\end{itemize}

By Dedekind's modular law for groups (see for example \cite[1.1.1]{KuSt}), every normal subgroup of a group $G$ is modular in $\cl(G)$.  Note that if $m$ is modular in $L$ and some interval $[a,b]:=\{x \in L:a \leq x \leq b\}$ from $L$ contains $m$, then $m$ is also modular in $[a,b]$.  We will need the following fact about modular elements, which is well known and appears, for example, as Theorem 2.1.6(d) in \cite{Sch2}.

\begin{lemma} \label{modjoin}
If $m,n$ are modular elements in the lattice $L$, then $m \vee n$ is modular in $L$.
\end{lemma}

Most of our group theoretic notation follows that in \cite{Sch2}.  The center of a group $G$ will be denoted by $Z(G)$.  For a prime $p$, $O_p(G)$ is the largest normal $p$-subgroup of $G$.  For $H \leq G$, $H_G$ will denote the core of $H$ in $G$, that is, the intersection of all $G$-conjugates of $H$.  Also, $H^G$ will denote the normal closure of $H$ in $G$, that is, the subgroup generated by all $G$-conjugates of $H$, and $C_G(H)$ will denote the centralizer of $H$ in $G$.  Both $H_G$ and $H^G$ are normal in $G$.  A subgroup $M \leq G$ is {\it permutable} in $G$ if $HM=MH$ for all $H \leq G$.  A {\it $P$-group} (not to be confused with a $p$-group) is a group $H=CE$, not of prime order, such that 
\begin{itemize}
\item $E$ is a nontrivial normal elementary abelian $p$-subgroup of $H$ for some prime $p$, 
\item $C$ is cyclic, either trivial or of prime order $q \neq p$, and 
\item for each non-identity $c \in C$, there is some positive integer $n=n(c)$ with $2 \leq n \leq p-1$ such that $c^{-1}xc=x^n$ for all $x \in E$.
\end{itemize}
If $G$ is a finite group and $H,K$ are normal subgroups of $G$ with $K \leq H$, we say $H/K$ is {\it hypercentrally embedded} in $G$ if there exists a chain
\[
K=N_0 \leq N_1 \leq \ldots \leq N_r=H
\]
of subgroups of $G$ such that $[N_i,G] \leq N_{i-1}$ for all $i \in [r]$.  We say $H$ is hypercentrally embedded in $G$ if $H/1$ is hypercentrally embedded.  A {\it minimal normal subgroup} of $G$ is a nontrivial normal subgroup not containing properly any nontrivial normal subgroup.  Every minimal normal subgroup $N$ of $G$ is characteristically simple, that is, no nontrivial proper subgroup of $N$ is invariant under the action of $\operatorname{Aut}(N)$.

\section{The proof of Theorem \ref{nonsolv}} \label{nsproof}

There are two key facts from group theory used in the proof of Theorem \ref{nonsolv}.  We present these facts below as Lemmas \ref{autchs} and \ref{minnor}.  Lemma \ref{autchs} is weaker than \cite[Lemma 3.24]{BaLu} (see also \cite[6.6.3(c)]{KuSt}), the proof of which uses what is essentially Lemma \ref{minnor}.  The rest of the proof of Theorem \ref{nonsolv} is elementary.

\begin{lemma} \label{autchs}
Let $N$ be a nonabelian, characteristically simple finite group and let $A$ be a solvable group of automorphisms of $N$.  Then $A$ fixes (setwise) some nontrivial proper subgroup of $N$.
\end{lemma}

\begin{proof}
See \cite[Lemma 3.24]{BaLu} or \cite[6.6.3(c)]{KuSt}
\end{proof}

\begin{lemma} \label{minnor}
Let $M$ be a maximal subgroup of the finite group $G$, and let $K$ be a nonabelian minimal normal subgroup of $G$ that is not contained in $M$.  Then $M \cap K$is not a nontrivial abelian $p$-group.
\end{lemma}

\begin{proof}

Assume for contradiction that $M\cap K$ is a nontrivial abelian $p$-group.  We observe that $K$ is characteristically simple, and is not a $p$-group.  It follows that $O_p(K) = 1$, as otherwise it would be a nontrivial proper characteristic subgroup of $K$.  In particular we have $N_K(M \cap K) < K$.


If $M \cap K$ is a self-normalizing Sylow $p$-subgroup of $K$, then, by Burnside's Normal $p$-complement Theorem (see for example \cite[Lemma 7.2.1]{KuSt}), $K$ has a normal Hall $p^\prime$-subgroup $X$.  Now, as $1<M \cap K<K$, we see that $X$ is a nontrivial proper characteristic subgroup of $K$, contradicting the fact that $K$ is minimal normal in $G$.  Therefore, if $M \cap K$ is a Sylow $p$-subgroup of $K$ then $M \cap K<N_K(M \cap K)$.

Assume $M \cap K$ is not a Sylow $p$-subgroup of $K$. Let $P$ be a Sylow $p$-subgroup of $K$ such that $M \cap K<P$.  Then, as is well known, we have $M \cap K<N_P(M \cap K)$.

Combining the results we have obtained so far, we see that  
\[
M \cap K<N_K(M \cap K)<K.
\]
Since $M \cap K$ and $K$ are $M$-invariant, so is $N_K(M \cap K)$.  Now an easy order argument gives
\[
M<MN_K(M \cap K)<MK=G,
\]
contradicting the maximality of $M$.
\end{proof}


For a group $G$ and a normal subgroup $N$ of $G$, $\nl_G(N)$ will denote the largest number $t$ such that there exists a chain
\[
1=N_0 \lhd \ldots \lhd N_t=N
\]
of length $t$ consisting of normal subgroups of $G$.  Note that
\begin{equation} \label{chn}
\chl(G)=\chl(G/N)+\nl_G(N)
\end{equation}

\begin{lemma} \label{cldiff}
Let $M$ be a maximal subgroup of the finite group $G$.  Choose $N \lhd G$ such that $N/M_G$ is a minimal normal subgroup of $G/M_G$. Then
\begin{equation*} 
\chl(M)-\chl(G)=\nl_M(M_G)-\nl_G(M_G)+\nl_{M/M_G}((M \cap N)/M_G)-1.
\end{equation*}
\end{lemma}

\begin{proof}
Since $G=MN$, we have $G/N \cong M/(M \cap N)$.  Now
\\
\\
$\chl(M)-\chl(G)$
\begin{eqnarray*}
& = & \nl_M(M_G)-\nl_G(M_G)+\chl(M/M_G)-\chl(G/M_G) \\ & = & \nl_M(M_G)-\nl_G(M_G)+\chl(M/M_G)-(1+\chl(G/N)) \\ & = & \nl_M(M_G)-\nl_G(M_G)+\chl(M/M_G)-\chl(M/(M \cap N))-1 \\ & = & \nl_M(M_G)-\nl_G(M_G)+\nl_{M/M_G}((M \cap N)/M_G)-1
\end{eqnarray*}
\end{proof}

Note that whenever $H \leq G$, we have
\[
\nl_G(H_G) \leq \nl_H(H_G).
\]
The next result, which has significant overlap with \cite[Proposition 2.3]{HaSo} follows immediately.

\begin{cor} \label{cdcor}
Let $G,M,N$ be as in Lemma \ref{cldiff}.
\begin{enumerate}
\item We have $\chl(M) \geq \chl(G)-1$.
\item If $M \cap N \neq M_G$, then $\chl(M) \geq \chl(G)$.
\item If $(M \cap N)/M_G$ is neither trivial nor a minimal normal subgroup of $M/M_G$, then $\chl(M) \geq \chl(G)+1$.
\end{enumerate}
\end{cor}

The next corollary of Lemma \ref{cldiff} also follows from the fact, mentioned in the introduction, that every maximal chain in a lattice $\cl$ is at least as long as every chain of modular elements in $\cl$.

\begin{cor} \label{mmlgechl}
Let $G$ be any finite group.  Then
\[
\mml(G) \geq \chl(G).
\]
\end{cor}

\begin{proof}
We proceed by induction on $|G|$, the case $G=1$ being trivial.  Assume now that $|G|>1$ and let
\[
1=M_0<\ldots<M_r=G
\]
be a maximal chain in $\cl(G)$.  Then
\[
r \geq 1+\mml(M_{r-1}) \geq 1+\chl(M_{r-1}) \geq \chl(G),
\]
the second inequality following from the inductive hypothesis and the third from Corollary \ref{cdcor}(1).
\end{proof}

\begin{prop} \label{chlesns}
Let $G$ be a nonsolvable finite group and let $M$ be a solvable maximal subgroup of $G$.  Then $\chl(M)>\chl(G)$.
\end{prop}

\begin{proof}
Let $N$ be as in Lemma \ref{cldiff}.  We will apply Corollary \ref{cdcor}(3).  We assume without loss of generality that $M_G=1$.  Since $M$ is solvable and $G$ is not, we see that $N$ is not solvable.  Thus $N$ is a nonabelian characteristically simple group.  By Lemma \ref{autchs} and the maximality of $M$, $M \cap N \neq 1$.  Since $M$ is solvable, we know that every minimal normal subgroup of $M$ is an elementary abelian $p$-group.  By Lemma \ref{minnor}, $M \cap N$ is not a minimal normal subgroup of $M$.  Thus Corollary \ref{cdcor}(3) applies.
\end{proof}

\begin{proof}[Proof of Theorem \ref{nonsolv}.]  Assume for contradiction that $G$ is a counterexample to the claim of Theorem \ref{nonsolv} with $|G|$ minimal.  Let
\[
1=M_0<\ldots<M_r=G
\]
be a maximal chain in $\cl(G)$ with $r<\chl(G)+2$.  If $M_{r-1}$ is not solvable then we obtain the contradiction
\[
r-1 \geq \chl(M_{r-1})+2 \geq \chl(G)+1,
\]
the first inequality following from the minimality of $|G|$ and the second from Corollary \ref{cdcor}(1).  If $M_{r-1}$ is solvable then we obtain the contradiction
\[
r-1 \geq \chl(M_{r-1}) \geq \chl(G)+1,
\]
the first inequality following from Corollary \ref{mmlgechl} and the second from Corollary \ref{chlesns}.
\end{proof}

\section{Proof of Theorem \ref{mod}} \label{modproof}

Our proof of Theorem \ref{mod} relies on four results from \cite{Sch2}.  We list these results below, making only minor notational changes from the statements given in \cite{Sch2}.

\begin{lemma}[Lemma 5.1.9 from \cite{Sch2}] \label{519}
If $M$ is a subgroup of prime power order of the finite group $G$, then the following properties are equivalent.
\begin{itemize}
\item[(a)] $M$ is modular in $\cl(G)$.
\item[(b)] $M$ is modular in $\cl(\langle M,x \rangle)$ for all $x \in G$ of prime power order.
\item[(c)] $M$ is permutable in $G$ or $G/M_G=M^G/M_G \times K/M_G$, where $M^G/M_G$ is a nonabelian $P$-group of order prime to $|K/M_G|$.
\end{itemize}
\end{lemma}

\begin{lemma}[Lemma 5.1.12 from \cite{Sch2}] \label{512}
Let $G$ be a finite group and $M \leq G$ such that $M$ is modular in $\cl(\langle M,x \rangle)$ for every $x \in G$ of prime power order.  If $Q/M_G$ is a Sylow subgroup of $M/M_G$, then $Q$ is modular in $\cl(G)$.
\end{lemma}

\begin{lemma}[Lemma 5.2.2 from \cite{Sch2}] \label{522}
If $p$ is a prime and $N$ a normal $p$-subgroup of the finite group $G$, then $N$ is hypercentrally embedded in $G$ if and only if $G/C_G(N)$ is a $p$-group.
\end{lemma}

\begin{thm}[Theorem 5.2.3 from \cite{Sch2}] \label{523}
If $M$ is permutable in the finite group $G$, then $M^G/M_G$ is hypercentrally embedded in $G$.
\end{thm}

The next lemma is the key result in our proof of Theorem \ref{mod}.

\begin{lemma} \label{minmod}
Say $M \leq G$ is nontrivial and modular in $\cl(G)$, and no nontrivial modular element of $\cl(G)$ is properly contained in $M$.  If $M$ is not a minimal normal subgroup of $G$ then $M^G$ contains a subgroup of prime order that is normal in $G$.
\end{lemma}

\begin{proof}
If $M \lhd G$ then $M$ is a minimal normal subgroup of $G$.  Assume that $M$ is not normal in $G$.  Since $M_G<M$, we have $M_G=1$.  Certainly $M$ is a modular element of every interval of $\cl(G)$ containing $M$.  In particular, $M$ is modular in $\cl(\langle M,x \rangle)$ for every $x \in G$ having prime power order.  By Lemma \ref{512}, every Sylow subgroup of $M$ is modular in $\cl(G)$.  Therefore, $M$ is a $q$-group for some prime $q$.  By Lemma \ref{519}, either
\begin{itemize}
\item[(a)] $M$ is permutable in $G$, or
\item[(b)] $G=M^G \times K$, where $M^G$ is a nonabelian $P$-group of order coprime with $|K|$.
\end{itemize}

Assume first that (a) holds.  Then every $G$-conjugate of $M$ is permutable in $G$.  It follows that $M^G$ is the setwise product of the $G$-conjugates of $M$ and therefore is a normal $q$-subgroup of $G$.  Let $Q$ be a Sylow $q$-subgroup of $G$.  Then $M^G \lhd Q$ and, as is well known, $M^G \cap Z(Q) \neq 1$.  Let $X \leq M^G \cap Z(Q)$ have order $q$.  By Theorem \ref{523}, $M^G$ is hypercentrally embedded in $G$.  Now, by Lemma \ref{522}, $G/C_G(M^G)$ is a $q$-group.  Since $C_G(M^G) \leq C_G(X)$, we see that $[G:C_G(X)]$ is a power of $q$.  On the other hand, the Sylow $q$-subgroup $Q$ is contained in $C_G(X)$.  Thus we must have $C_G(X)=G$, so $X \lhd G$.

Now assume that (b) holds.  Write $M^G=CE$, where $E$ is elementary abelian and $C$ is generated by an element conjugating every $x \in E$ to some non-trivial power of $x$.  Let $X \leq E$ have prime order $p$.  Certainly $X \lhd E$.  Also, $C$ normalizes $X$.  Finally, since $K$ commutes with $M^G$, we see that $K$ normalizes $X$, so $X \lhd G$.
\end{proof}

\begin{proof}[Proof of Theorem \ref{mod}] We prove Theorem \ref{mod} by induction on $r=\chl(G)$.  When $r=0$ we have $G=1$ and there is nothing to prove.  When $r=1$, $G$ is simple and the claim of the theorem follows from the fact that in this case the only modular elements of $\cl(G)$ are $1$ and $G$ (this is Theorem 5.3.1 in \cite{Sch2} and also follows immediately from Lemma \ref{minmod}).  Assume now that $r>1$.  Let
\[
1=M_0<M_1<\ldots<M_t=G
\]
be a chain of modular elements in $\cl(G)$ that is properly contained in no other chain of modular elements.  We wish to show that $t \leq r$.  Assume first that $M_1 \lhd G$. Then $M_1$ is a minimal normal subgroup of $G$ and $\chl(G/M_1)=r-1$.  Also, the interval $[M_1,G]$ in $\cl(G)$ is isomorphic with $\cl(G/M_1)$.  Since each $M_i$ ($i \geq 1$) is a modular element of $[M_1,G]$, we have $t-1 \leq r-1$ by inductive hypothesis.

Assume now that $M_1$ is not normal in $G$.  By Lemma \ref{minmod}, there is some $X \leq M_1^G$ that is a normal subgroup of $G$ having prime order $p$.  We claim that the chain
\[
1<X=M_0X<M_1X\leq M_2X \leq \ldots \leq M_tX=G
\]
contains $t+1$ distinct elements, all of which are modular in $\cl(G)$.  Indeed, modularity follows from Lemma \ref{modjoin}.  Consider the smallest $s$ such $X \leq M_s$.  Note $s>1$ by our assumption that our original chain of modular elements is properly contained in no larger such chain.  If $1 \leq j<s$ then
\[
|M_{j-1}X|=p|M_{j-1}|<p|M_j|=|M_jX|,
\]
so $M_{j-1}X<M_jX$.  If $j\geq s$ then $M_jX=M_j$, so if $j>s$ then $M_{j-1}X<M_jX$.  Thus our claim holds.  Now we can apply our inductive hypothesis as we did in the case $M_1 \lhd G$, using $X$ in place of $M_1$. 
\end{proof}

\section{Values of $\mml(G)-\chl(G)$} \label{chdiffex}

Our main result in this section is Theorem \ref{disc}, which shows that $\mml(G)-\chl(G)$ can take any value in $\nn \setminus \{1\}$.  


\begin{lemma} \label{liusag}
Let $G$ be a group, let $N \lhd G$, let $B \leq G$ and let $A$ be a maximal subgroup of $B$.  Then either $AN=BN$ or $AN$ is a maximal subgroup of $BN$.
\end{lemma}

\begin{proof}
Suppose $AN \leq X \leq BN$.  Then $X \cap A$ is either $B$ or $A$.  If $X \cap B=B$ then $BN \leq X$ so $X=BN$.  If $X \cap B=A$ then
\[
AN=(X \cap B)N=X \cap BN=X,
\]
by the Dedekind modular law.
\end{proof}

%

\begin{lemma} \label{projdp}
Let $G$ be a finite group and let $1 \neq N \lhd G$.  Then $\mml(G)>\mml(G/N)$
\end{lemma}

\begin{proof}
Let
\[
1=M_0<\ldots<M_r=G
\]
be a maximal chain in $\cl(G)$.  For $0 \leq i \leq r$, set $\overline{M}_i=M_iN/N$.  By Lemma \ref{liusag}, we get a maximal chain in $\cl(G/N)$ upon appropriately erasing repeated terms from
\[
\overline{M}_0 \leq \ldots \leq \overline{M}_r.
\]
Find the smallest $j$ such that $N \leq M_j$.  Then $\overline{M}_{j-1}=\overline{M}_j$, so there is at least one repeated term to erase.  Thus $\mml(G/N)<r$.
\end{proof}

Before continuing, we remark that the statement analogous to Lemma \ref{projdp}, concerning maximal chains in an arbitrary finite lattice $L$ and saturated chains starting at a (left) modular element $\hat{0} \neq m \in L$, admits an analogous proof.  We note also that Corollary \ref{mmlgechl} follows directly from Lemma \ref{projdp}.



\begin{lemma} \label{dpsimp}
Let $G$ be a finite group and let $S$ be a nonabelian finite simple group such that no section of $G$ is isomorphic with $S$.  Then
\[
\mml(G \times S) \geq \mml(G)+2.
\]
\end{lemma}

\begin{proof}
We proceed by induction on $|G|$, the claim in the case $G=1$ being that a nonabelian simple group has a nontrivial proper subgroup.  Now assume $|G|>1$ and let
\[
1=M_0<\ldots<M_r=G \times S
\]
be a maximal chain in $\cl(G \times S)$ with $\mml(G \times S)=r$.  Let $\pi:G \times S \rightarrow G$ map $(g,s)$ to $g$.  Let $J=\pi(M_{r-1})$.  If $J \neq G$ then, since  $M_{r-1}\leq J \times S$ is maximal in $G \times S$, we have $M_{r-1}=J \times S$.  Our inductive hypothesis applies, and we get
\[
\mml(G \times S)=1+\mml(J \times S) \geq 3+\mml(J) \geq 2+\mml(G).
\]
If $J=G$ then, since $G$ has no quotient isomorphic with $S$, we must have $M_{r-1}=G \times H$ for some maximal $H<S$ (see, for example, \cite[Lemma 1.3]{The}).  Since $S$ is nonabelian simple, we have $H \neq 1$.  Now
\[
\mml(G \times S)=1+\mml(G \times H)\geq 2+\mml(G),
\]
the last inequality following from Lemma \ref{projdp}.
\end{proof}


\begin{thm} \label{disc}
For every $n \in \nn \setminus \{1\}$, there exists some finite group $G$ such that $\mml(G)-\chl(G)=n$.
\end{thm}

\begin{proof}
We will use the following facts about the projective special linear group $L_2(p)$ over a field of prime order $p$. These facts can be found in \cite{Di}.
\begin{enumerate}
\item If $p>3$ then $L_2(p)$ is simple.  \item If $p$ is odd then $|L_2(p)|=p(p^2-1)/2$.  \item Every maximal subgroup of $A_5 \cong L_2(5)$ is isomorphic with one of $S_3$, $D_{10}$ or $A_4$.  \item Every maximal subgroup of $L_2(31)$ is isomorphic with one of $A_5$, $S_4$,$D_{30}$, $D_{32}$ or $\zz_{31}.\zz_{15}$. \item If $p \equiv 1 \bmod 5$ then $L_2(p)$ has a maximal subgroup isomorphic with $A_5$, and every subgroup of $L_2(p)$ that is isomorphic with $A_5$ is maximal.
\end{enumerate}

It is straightforward to check that each of $S_3,D_{10}$ and $A_4$ is solvable with chief length two.  It follows from fact (3) above and Theorem \ref{solv} that $\mml(A_5)=3$, so 
\begin{equation} \label{a5}
\mml(A_5)-\chl(A_5)=2.
\end{equation}
Similarly, each of $S_4,D_{30},D_{32}$ and $\zz_{31}.\zz_{15}$ is solvable, and these groups have respective chief lengths $3,3,5$ and $3$.  Combining these facts with Theorem \ref{solv}, fact (4) and $\mml(A_5)=3$, we get $\mml(L_2(31))=4$ and
\begin{equation} \label{l231}
\mml(L_2(31))-\chl(L_2(31))=3.
\end{equation}
Now let $p_1=31,p_2,\ldots$ be an infinite sequence of primes such that $p_i<p_{i+1}$ and $p_i \equiv 1 \bmod 5$ for all $i \geq 1$.  (The existence of such a sequence is guaranteed by the famous theorem of Dirichlet, see for example \cite{IrRo}.)  For $n \in \nn$, set
\[
G_n:=\prod_{i=1}^{n}L_2(p_i).
\]
We claim that
\begin{equation} \label{gn}
\mml(G_n)=2n+2
\end{equation}
for all $n \in \nn$.  Since $\chl(G_n)=n$, (\ref{gn}), together with (\ref{a5}), proves Theorem \ref{disc}.  We prove (\ref{gn}) by induction on $n$, the case $n=1$ being (\ref{l231}).  Now assume $n>1$.  By fact (2), $|G_{n-1}|$ is not divisible by $p_n$.  Applying (2) again, we see that $G_{n-1}$ has no section isomorphic with $L_2(p_n)$.  Combining fact (1) and Lemma \ref{dpsimp} with our inductive hypothesis, we get
\[
\mml(G_n) \geq 2n+2.
\]
It remains to exhibit a maximal chain of length $2n+2$ in $\cl(G_n)$.

For each $i \in [n]$, fix injective homomorphisms $\rho_i:\zz_3 \rightarrow L_2(p_i)$, $\sigma_i:A_4 \rightarrow L_2(p_i)$ and $\tau_i:A_5 \rightarrow L_2(p_i)$, such that
\[
{\rm Image}(\rho_i)<{\rm Image}(\sigma_i)<T_i:={\rm Image}(\tau_i).
\]
(The existence of such homomorphisms is guaranteed by fact (5).)  Set
\begin{eqnarray*}
D_1 & := & \{(\rho_1(x),\ldots,\rho_n(x)):x \in \zz_3\}, \\ D_2 & := & \{(\sigma_1(x),\ldots,\sigma_n(x)):x \in A_4\}, \\ D_3 & := & \{(\tau_1(x),\ldots,\tau_n(x)):x \in A_5\}.
\end{eqnarray*}
It is straightforward to confirm that $D_1 \cong \zz_3$ is a maximal subgroup of $D_2 \cong A_4$, which is in turn a maximal subgroup of $D_3 \cong A_5$.  Now set
\[
T:=\prod_{i=1}^{n}T_i,
\]
and for $0 \leq k \leq n$, set
\[
T(k):=\{(t_1,\ldots,t_n) \in T:\tau_i^{-1}(t_i)=\tau_j^{-1}(t_j) \mbox{ for }1 \leq i<j \leq k\}.
\]
Then
\[
D_3=T(n)<T(n-1)<\ldots<T(1)=T.
\]
Note that $T(k)$ is a maximal subgroup of $T(k-1)$ for all $2 \leq k \leq n$.  Indeed, if $T(k) \leq X \leq T(k-1)$, the natural projection of $X$ to $L_2(p_k)$ has image $T_k$.  Since $A_5$ is simple, the subgroup $Y$ of $X$ consisting of all $(t_1,\ldots,t_n) \in X$ such that $t_j=1$ for $j \neq k$ must be either $1$ or $T_k$.  In the first case, we have $X=T(k)$ and in the second case we have $X=T(k-1)$.  Finally, for $0 \leq k \leq n$, set
\[
L(k):=\{(x_1,\ldots,x_n) \in  G_n:x_i \in T_i \mbox{ for } 1 \leq i \leq k\}.
\]
We have
\[
T=L(n)<L(n-1)<\ldots<L(0)=G_n.
\]
Since $T_i$ is maximal in $L_2(p_i)$ for all $i \in [n]$ (fact (5)), we see that $L(k)$ is maximal in $L(k-1)$ for all $k \in [n]$.  We have now that
\[
1<D_1<D_2<D_3=T(n)<\ldots<T(1)=L(n)<\ldots<L(0)=G_n
\]
is a maximal chain of length $2n+2$.
\end{proof}

\end{document}